\documentclass[english, 12pt]{article}
\usepackage{amsfonts}
\usepackage{amsmath}
\usepackage{amssymb}
\usepackage{amsthm}
\usepackage{amscd}
\usepackage{amscd}
\usepackage{color}
\usepackage{amsthm}
\usepackage{textcomp}
\newtheorem{defi}{Definition}[section]

\newtheorem{theo}{Theorem}[section]
\newtheorem{prop}{Proposition}[section]

\newtheorem{rk}{Remark}[section]

\newtheorem{exa}{Example}[section]

\numberwithin{equation}{section}

\def\R{\mathbb{R}}

\def\C{\mathbb{C}}
\def\B{\mathbb{B}}
\def\Z{\mathbb{Z}}
\def\D{\mathbb{D}}
\def\Q{\mathbb{Q}}
\def\Qh{\mathcal{Q}_h}

\def\ZB{\mathcal{Z}}

\def\P{ \mathcal{P}}

\def\fraka{\mathfrak{a}}
\def\frakp{\mathfrak{p}}

\newcommand{\Rj}{\mathrm{Re}}

\newcommand{\eg}{\mathbf{e}}

\newcommand{\ug}{\mathbf{u}}

\newcommand{\ig}{\mathbf{i}}
\newcommand{\jg}{\mathbf{j}}
\newcommand{\kg}{\mathbf{k}}

\begin{document}
\bibliographystyle{alpha}

\title{  Bicomplex Algebraic Numbers}
\author{\normalsize Hichem Gargoubi\footnote {Universit\'e de Tunis, I.P.E.I.T., Department of Mathematics, 2 Rue Jawaher Lel Nehru, Monfleury, Tunis, 1008 Tunisia. E-mail address: hichem.gargoubi@ipeit.rnu.tn}
	\, and \, Sayed Kossentini\footnote {Universit\'e de Tunis El Manar, Facult\'e des Sciences de Tunis, Department of Mathematics, Lab, Topology, Algebra, Arithmetic \& Order,  2092, Tunis,  Tunisia. E-mail address:
	sayed.kossentini@fst.utm.tn}\footnote {Universit\'e de Jendouba, Institut National des Technologies et des Sciences du Kef, 7100, Le Kef, Tunisia }}

\date{}
\maketitle
\begin{abstract}

	We investigate bicomplex analogues of fundamental notions from classical algebraic number theory. In particular, we show that the primitive element theorem admits a natural generalization to bicomplex extensions, giving rise to two distinct classes of extensions depending on the nature of the generating element. We further establish a key decomposition property for bicomplex extensions, which serves as a foundation for studying their rings of integers. We also observe that prime elements in the ring of integers of a number field may become semiprime in the rings of integers of suitable bicomplex extensions. Finally, we present two explicit examples of finite bicomplex extensions.

\end{abstract}
{\bf Key words}: algebraic number fields, hyperbolic numbers,   bicomplex numbers

\section{Introduction}
The concept of an algebraic number field originates in arithmetic questions concerning the integers. 
This naturally leads to enlargements of the classical domain of arithmetic to rings such as the Gaussian integers $\mathbb{Z}[\ig]$, the Eisenstein integers $\mathbb{Z}\!\left[\tfrac{1}{2}(-1+\ig\sqrt{3})\right]$, or more generally the cyclotomic integers $\mathbb{Z}[e^{2\pi\ig/n}]$. 
These rings play a central role in many areas, for example in reciprocity laws, heights, and Fermat’s Last Theorem.
The modern formulation of this theory goes back to Dedekind, who introduced the ring of integers $\mathcal{O}_K$ of a number field $K$ and established unique factorization in terms of ideals; for historical developments see~ e.g. \cite{Ed}.
The theory of number fields has since grown into a broad and active area of research intersecting several branches of mathematics.\\

In this article we investigate algebraicity in the setting of the bicomplex numbers. 
This commutative four-dimensional real algebra was introduced by Segre in 1892 \cite{SER}; it is generated by three imaginary units $\ig,\jg,\kg$ with relations
$$
\ig^{2}=\kg^{2}=-1,\qquad \jg^{2}=1,\qquad \ig\jg=\jg\ig=\kg.
$$
Among complex Clifford algebras, the bicomplex algebra is the only commutative one. 
Although it is not a division algebra—unlike the quaternion algebra of Hamilton \cite{Hamil}—its algebraic and analytic structure has been extensively developed, with applications in mathematics and quantum physics (see, e.g., \cite{Ham,Krav,Price,ZET,DRQ1,DRQ2,C*alg,COLOM1,LUNA1}).\\

The bicomplex algebra simultaneously extends two classical number systems: the complex numbers 
$$
\mathbb{C}\cong \mathrm{Cl}_{\mathbb{R}}(0,1)=\mathbb{R}[X]/(X^2+1)
$$
and the hyperbolic numbers (also called duplex or split numbers)
$$
\mathbb{D}\cong \mathrm{Cl}_{\mathbb{R}}(1,0)=\mathbb{R}[X]/(X^2-1).
$$
Both are extensions of $\mathbb{R}$, but from the perspective of lattice-ordered algebras ($\ell$-algebras) the hyperbolic numbers form, up to isomorphism, the unique real associative planar algebra admitting a natural Archimedean $f$\nobreakdash-algebra  structure (see \cite{GK1}). 
The presence of two idempotent zero divisors also yields an efficient decomposition that simplifies many arguments.\\

Our main objective is to study finite bicomplex extensions $L$, viewed as finite $\mathbb{Q}$-algebras contained in the bicomplex algebra. 
We prove that every such extension is simple, i.e.,
$$
L=\mathbb{Q}[\omega]
$$
for some $\omega\in L$. A careful analysis of $\omega$ provides necessary and sufficient conditions for $L$ to be a (classical) number field. 
This viewpoint naturally produces new families of algebraic integers beyond the standard ones and clarifies several structural phenomena. 
For instance, we show that bicomplex extensions lead to exactly two types of quadratic extensions of $\mathbb{Q}$: the classical quadratic fields $\mathbb{Q}(\sqrt{d})$ and the hyperbolic extension $\mathbb{Q}[\mathrm{j}]$, which is the only quadratic extension of $\mathbb{Q}$ into an Archimedean $f$-ring \cite{GK3}. 
We also observe that primes in the ring of integers of a number field may become semiprime in the ring of integers of suitable bicomplex extensions.\\

The article is organized as follows. 
Section~2 reviews the basic algebraic and structural properties of bicomplex numbers. 
Section~3 introduces bicomplex algebraic numbers and their minimal polynomials. 
Section~4 develops the theory of finite bicomplex extensions and establishes the simplicity result $L=\mathbb{Q}[\omega]$, together with criteria for $L$ to be a number field; we also describe the corresponding ring of integers and introduce the Dedekind-type zeta function of a bicomplex extension. 
Section~5 classifies quadratic bicomplex extensions into two families and exhibits a distinguished class of quartic extension containing the classical imaginary quadratic field $\mathbb{Q}(\mathrm{i})$.

\section{Bicomplex numbers}
In this section we briefly recall the basic properties of the bicomplex algebra that will be used throughout the paper (see, e.g., \cite{DRA}). 
The algebra of bicomplex numbers is defined by 
\begin{equation}\label{realdecomp}
\mathbb{B}
:=\{\, x+y\ig+z\jg+t\kg : x,y,z,t\in\mathbb{R} \,\},
\end{equation}
where the imaginary units $\ig,\jg,\kg\notin\mathbb{R}$ satisfy
$$
\ig^{2}=\kg^{2}=-1,\qquad 
\jg^{2}=1,\qquad 
\ig\jg=\jg\ig=\kg.
$$
The algebra $\mathbb{B}$ contains three real two-dimensional subalgebras
$$
\mathbb{C}_{\ug}:=\{\,x+\ug y : x,y\in\mathbb{R}\,\}, \qquad \ug\in\{\ig,\jg,\kg\},
$$
namely the field of complex numbers $\mathbb{C}_{\ig}=\mathbb{C}$, its copy
$\mathbb{C}_{\kg}$, and the hyperbolic (or duplex) numbers
$\mathbb{C}_{\jg}=\mathbb{D}$.  
Accordingly, $\mathbb{B}$ can be regarded as a complexification of both the complex and the hyperbolic numbers, namely,
\begin{eqnarray*}\label{complexfdecomo}
\B= \C + \kg \C= \D+ \ig \D.
\end{eqnarray*}
Among these subalgebras, the hyperbolic numbers $\mathbb{D}$  possess a basis of idempotent zero divisors.  
Indeed, setting
\begin{equation}\label{idempbasis}
\mathbf{e}_{1}:=\frac{1+\jg}{2}, 
\qquad 
\mathbf{e}_{2}:=\frac{1-\jg}{2},
\end{equation}
we have
$$
\mathbf{e}_{1}+\mathbf{e}_{2}=1,\qquad 
\mathbf{e}_{1}^{2}=\mathbf{e}_{1},\qquad 
\mathbf{e}_{2}^{2}=\mathbf{e}_{2},\qquad 
\mathbf{e}_{1}\mathbf{e}_{2}=0.
$$
Using the idempotent basis \eqref{idempbasis}, every bicomplex number $\omega=z_{1}+\kg z_{2}$ admits the unique decomposition
\begin{equation}\label{idempdecomp}
\omega=\mathcal{P}_{1}(\omega)\mathbf{e}_{1}+\mathcal{P}_{2}(\omega)\mathbf{e}_{2},
\end{equation}
where the maps  $\mathcal{P}_k: \B\longrightarrow \C$ are  surjective  algebra homomorphisms, defined by
\begin{equation}\label{mP}
\P_1(\omega)= z_1-\ig z_2 \quad\mbox{and}\quad \P_2(\omega)= z_1+\ig z_2.
\end{equation}
Thus $\mathbb{B}$ is a two-dimensional $\mathbb{C}$-algebra with idempotent basis $(\mathbf{e}_{1},\mathbf{e}_{2})$, and all algebraic operations are componentwise.  
In particular, the group of units is
$$
\mathbb{B}^{\times}
=\{\,\omega\in\mathbb{B} : \mathcal{P}_{1}(\omega)\neq 0,\ \mathcal{P}_{2}(\omega)\neq 0\,\}
=\mathbb{C}^{\times}\mathbf{e}_{1}+\mathbb{C}^{\times}\mathbf{e}_{2}.
$$There are three natural conjugations on $\mathbb{B}$, defined for 
$\omega=x+y\ig+z\jg+t\kg$ by
\begin{equation}\label{conjugations}
\overline{\omega}^{\,\ig}= x+y\ig-z\jg-t\kg,\qquad
\overline{\omega}^{\,\jg}= x-y\ig+z\jg-t\kg,\qquad
\overline{\omega}^{\,\kg}= x-y\ig-z\jg+t\kg.
\end{equation}
Together with the identity, these conjugations form an abelian group of algebra automorphisms of $\mathbb{B}$.  
The product of $\omega$ with its three conjugates is a positive real number called the \emph{norm} of $\omega$:
\begin{equation}
N(\omega):= \omega\, \overline{\omega}^\ig \,\overline{\omega}^\jg\, \overline{\omega}^\kg= |\gamma_1\gamma_2|^2.
\end{equation}
where $\omega=\gamma_1\eg_2+ \gamma_2\eg_2$ is the idempotent representation.  
The function $N:\mathbb{B}\to\mathbb{R}_{\geq 0}$ is multiplicative, and its zero set is precisely the set of all non-invertible elements of $\mathbb{B}$ called the \emph{null cone} of $\mathbb{B}$.

 
\section[Basic properties]{Bicomplex algebraic numbers}

In this section we establish some basic properties of bicomplex algebraic numbers and their minimal polynomials.  
For later use, note that a direct computation based on the properties of the idempotent basis \eqref{idempbasis} yields the following polynomial decomposition property:
\begin{equation}\label{decompoly}
P(\gamma_1 \eg_1+\gamma_2 \eg_2)=P(\gamma_1)\eg_1+P(\gamma_2)\eg_2, 
\quad \mbox{for all } P\in\C[X] \mbox{ and all } \gamma_1,\gamma_2\in\C.
\end{equation}

\subsection{Characterization}

We begin with a characterization of the set of bicomplex algebraic numbers, denoted by $\mathcal{A}_\B$.  
We also use the notation $\widehat{\Q}$ for the algebraic closure of $\Q$, i.e., the field of all complex algebraic numbers, which is well known to be countable.

\begin{theo}\label{bicomalgebr}
The set $\mathcal{A}_\B$ of all bicomplex algebraic numbers decomposes as
\begin{equation*}
\mathcal{A}_\B=\widehat{\Q}\eg_1+\widehat{\Q}\eg_2.
\end{equation*}
It is a countable set and is closed under all conjugation operators, i.e.,
$$
\overline{\omega}^{\ig},\ \overline{\omega}^{\jg},\ \overline{\omega}^{\kg} \in \mathcal{A}_\B 
\quad \mbox{for every } \omega\in\mathcal{A}_\B .
$$
\end{theo}

\begin{proof}
Let $\omega=\gamma_1\eg_1+\gamma_2\eg_2$ be a bicomplex number.  
If $\omega\in\mathcal{A}_\B$, then there exists a nonzero polynomial $P\in\Z[X]$ such that $P(\omega)=0$.  
Using \eqref{decompoly}, we obtain $P(\gamma_1)=P(\gamma_2)=0$, hence $\gamma_1,\gamma_2\in\widehat{\Q}$.
\\Conversely, assume that $\gamma_1,\gamma_2\in\widehat{\Q}$.  
Choose nonzero polynomials $P_1,P_2\in\Z[X]$ such that $P_1(\gamma_1)=0$ and $P_2(\gamma_2)=0$.  
Then \eqref{decompoly} implies that $\omega$ is a zero of the polynomial $P=P_1P_2$, and thus $\omega\in\mathcal{A}_\B$.  
Therefore
$$
\mathcal{A}_\B=\widehat{\Q}\eg_1+\widehat{\Q}\eg_2 .
$$
Since $\widehat{\Q}$ is countable and $\mathcal{A}_\B$ is in bijection with $\widehat{\Q}\times\widehat{\Q}$ via the map 
$\gamma_1\eg_1+\gamma_2\eg_2\mapsto (\gamma_1,\gamma_2)$, the set $\mathcal{A}_\B$ is countable.
\\Finally, closure under conjugation follows from the fact that each conjugation map is an $\R$-algebra homomorphism, and therefore satisfies
$$
P(\overline{\omega}^{\ug})=\overline{P(\omega)}^{\ug},
\quad P\in\R[X],\ \ug\in\{\ig,\jg,\kg\},\ \omega\in\B.
$$
\end{proof}

Theorem \ref{bicomalgebr} yields several equivalent descriptions of bicomplex algebraic numbers depending on the chosen representation of $\B$.  
In the real form \eqref{realdecomp}, a bicomplex number
$$
\omega=x+\ig y+\jg z+\kg t
$$
is algebraic if and only if $x,y,z,t$ are real algebraic numbers.  
This follows from the ring structure of $\mathcal{A}_\B$, its stability under conjugation, and the identities
\begin{eqnarray*}
\frac{1}{4}\left(\omega+\overline{\omega}^{\ig}+\overline{\omega}^{\jg}+\overline{\omega}^{\kg}\right) &=& x,\\[0.2cm]
\frac{1}{4\ig}\left(\omega+\overline{\omega}^{\ig}-\overline{\omega}^{\jg}-\overline{\omega}^{\kg}\right) &=& y,\\[0.2cm]
\frac{1}{4\jg}\left(\omega-\overline{\omega}^{\ig}+\overline{\omega}^{\jg}-\overline{\omega}^{\kg}\right) &=& z,\\[0.2cm]
\frac{1}{4\kg}\left(\omega-\overline{\omega}^{\ig}-\overline{\omega}^{\jg}+\overline{\omega}^{\kg}\right) &=& t. 
\end{eqnarray*}


\subsection{Minimal polynomial of a bicomplex algebraic number}

We establish here several properties of minimal polynomials of bicomplex algebraic numbers. Throughout, $P_\omega$ denotes the minimal polynomial of the bicomplex algebraic number $\omega$. It is assumed to be primitive (i.e., its coefficients have greatest common divisor~$1$) and to have positive leading coefficient (the coefficient of $X^n$, $n=\deg P_\omega$).

We begin with a characterization of $P_\omega$ depending on whether $\omega$ belongs to $\C_{\ug}$ for some $\ug\in\{\ig,\kg\}$.

\begin{theo}\label{thpolyminimal}
Let $\omega=\gamma_1 \eg_1 + \gamma_2 \eg_2$ be a bicomplex algebraic number. Then
$$
P_\omega = P_{\gamma_1} = P_{\gamma_2}
\quad\text{if and only if}\quad
\omega \in \C_{\ug}\ \text{for some}\ \ug\in\{\ig,\kg\}.
$$
Otherwise,
$$
P_\omega = P_{\gamma_1}P_{\gamma_2}.
$$
\end{theo}

\begin{proof}
Let $\omega=\gamma_1\eg_1+\gamma_2\eg_2$ be a bicomplex algebraic number. By Theorem~\ref{bicomalgebr}, the components $\gamma_1$ and $\gamma_2$ are complex algebraic numbers.\\
Assume first that $P_\omega=P_{\gamma_1}=P_{\gamma_2}$, necessarily irreducible in $\Q[X]$. Then the isomorphism
$$
\Q[\omega] \simeq \Q[X]/(P_\omega)
$$
shows that $\Q[\omega]$ is a field. Hence $\Q[\omega]$ must be contained in $\C_{\ug}$ for some $\ug\in\{\ig,\kg\}$, which implies $\omega\in\C_{\ug}$.\\
Conversely, suppose $\omega\in\C_{\ug}$ for some $\ug\in\{\ig,\kg\}$. In this case, one has $\gamma_1=\gamma_2$ if $\ug=\ig$ and $\gamma_2=\overline{\gamma_1}$ if $\ug=\kg$. In both situations,
$$
P_{\gamma_1}=P_{\gamma_2}=P.
$$
By~\eqref{decompoly}, $P(\omega)=0$, hence $P_\omega$ divides $P$. As $P_\omega$ is primitive, we obtain $P_\omega=P$.\\
Now suppose $\omega\notin\C_{\ug}$ for $\ug\in\{\ig,\kg\}$. Then $P_{\gamma_1}$ and $P_{\gamma_2}$ are relatively prime. Using again~\eqref{decompoly}, we have
$
P_\omega(\gamma_1)=P_\omega(\gamma_2)=0,
$
so $P_{\gamma_1}\mid P_\omega$ and $P_{\gamma_2}\mid P_\omega$. Hence
$$
P_{\gamma_1}P_{\gamma_2}\mid P_\omega.
$$
Conversely, by~\eqref{decompoly},
$
(P_{\gamma_1}P_{\gamma_2})(\omega)=0,
$
so 
$$P_\omega\mid P_{\gamma_1}P_{\gamma_2}.$$
The minimal polynomial being primitive, we conclude
$$
P_\omega = P_{\gamma_1}P_{\gamma_2}.
$$
\end{proof}
As an immediate consequence, the classical complex identity $P_\gamma=P_{\bar{\gamma}}$ extends to the three bicomplex conjugations.

\begin{prop}
Let $\omega$ be a bicomplex algebraic number. Then $\omega$ and its conjugates $\overline\omega^{\ig}$, $\overline\omega^{\jg}$ and $\overline\omega^{\kg}$ have the same minimal polynomial.
\end{prop}

\begin{proof}
The conclusion follows from Theorem~\ref{thpolyminimal}, the property $P_\gamma=P_{\bar\gamma}$ in the complex case, and the fact that if $\omega=\gamma_1\eg_1+\gamma_2\eg_2$, then
$$
\overline{\omega}^{\ig}=\gamma_2\eg_1+\gamma_1\eg_2,\qquad
\overline{\omega}^{\jg}=\bar{\gamma_1}\eg_1+\bar{\gamma_2}\eg_2,\qquad
\overline{\omega}^{\kg}=\bar{\gamma_2}\eg_1+\bar{\gamma_1}\eg_2.
$$
\end{proof}

We now describe the complex and bicomplex roots of $P_\omega$ and establish its relation with the product
$$
\prod_{\xi\in\Lambda_\B(\omega)} (X-\xi),
$$
where $\Lambda_\B(\omega)$ denotes the set of all bicomplex roots of $P_\omega$. We denote by $r$ the number of real roots of $P_\omega$, and by $2s$ the number of non-real complex roots.

\begin{theo}\label{theroots}
Let $\omega$ be a bicomplex algebraic number of degree $n$. Then $P_\omega$ has $n$ distinct complex roots and $n^2$ distinct bicomplex roots, with
$$
	n^2=r+2( s_\ig+ s_\jg+ s_\kg)+4d,
$$
where $r$ is the number of real roots, $2s_\ug$ is the number of non-real roots lying in $\C_\ug$ $(\ug=\ig,\jg,\kg)$, and $4d$ is the number of roots lying in none of these three subalgebras. Moreover,
$$
2s_\jg = r(r-1),\qquad
2s_\ig = 2s_\kg = 2s,\qquad
4d = 4s(s+r-1).
$$
\end{theo}

\begin{proof}
Let $\omega=\gamma_1 \eg_1+\gamma_2 \eg_2$ be a bicomplex algebraic number of degree $n$. 
According to Theorem~\ref{thpolyminimal}, we have either
$P_\omega=P_{\gamma_1}=P_{\gamma_2}$ or
$P_\omega=P_{\gamma_1}P_{\gamma_2}$, where $P_{\gamma_1}$ and $P_{\gamma_2}$ are relatively prime polynomials, each having no repeated roots. 
This implies that $P_\omega$ has $n$ distinct complex roots.
For bicomplex roots, by \eqref{decompoly}, we have
\[
\Lambda_\B(\omega)=\Lambda(\omega)\eg_1+\Lambda(\omega)\eg_2,
\]
where $\Lambda(\omega)$ denotes the set of complex roots of $P_\omega$.
From this and the above observation, it follows that $P_\omega$ has $n^2$ distinct bicomplex roots.\\
Now observe that $\Lambda_\B(\omega)$ is the disjoint union
\begin{equation}\label{unionrootset}
\Lambda_\B(\omega)=\mathcal{R}\cup \mathcal{S}_\ig\cup \mathcal{S}_\jg\cup \mathcal{S}_\kg \cup \mathcal{D},
\end{equation}
where
\begin{eqnarray*}
\mathcal{R}
&=&
\bigl\{ \psi \in \Lambda_\B(\omega):\, \overline{\psi}^{\ug}=\psi \text{ for all } \ug \in \{\ig,\jg,\kg\} \bigr\}
= \bigcup_{\psi \in \mathcal{R}}\{\psi\},\\[1ex]
\mathcal{S}_\ug
&=&
\bigl\{ \psi \in \Lambda_\B(\omega):\, \overline{\psi}^{\ug}=\psi,\ \overline{\psi}\neq\psi \bigr\}
= \bigcup_{\psi \in \mathcal{S}_\ug}\{\psi,\bar{\psi}\},
\quad \ug=\ig,\jg,\kg,\\[1ex]
\mathcal{D}
&=&
\bigl\{ \psi \in \Lambda_\B(\omega):\, \overline{\psi}^{\ug}\neq\psi \text{ for all } \ug \in \{\ig,\jg,\kg\} \bigr\}
= \bigcup_{\psi \in \mathcal{D}}\{\psi,\overline{\psi}^{\ig},\overline{\psi}^{\jg},\overline{\psi}^{\kg}\}.
\end{eqnarray*}
Recall that a bicomplex number $\psi$ is real (resp.\ $\psi\in\C_\ug$ for some $\ug\in\{\ig,\jg,\kg\}$) if and only if
$\overline{\psi}^{\ug}=\psi$ for all $\ug\in\{\ig,\jg,\kg\}$ (resp.\ for some $\ug\in\{\ig,\jg,\kg\}$).
It follows that $\mathcal{R}$ is the set of real roots with $\#\mathcal{R}=r$,
$\mathcal{S}_\ug$ $(\ug=\ig,\jg,\kg)$ is the set of non-real roots belonging to $\C_\ug$ with $\#\mathcal{S}_\ug=2s_\ug$,
and $\mathcal{D}$ is the set of roots not belonging to $\C_\ug$ for any $\ug=\ig,\jg,\kg$, with $\#\mathcal{D}=4d$.
Therefore, from \eqref{unionrootset}, we obtain
\begin{equation}\label{eqroots}
n^2=r+2(s_\ig+s_\jg+s_\kg)+4d.
\end{equation}
We now evaluate $2s_\jg$, $2s_\ig$, $2s_\kg$, and $4d$ in terms of $(r,s)$.
For the set $\mathcal{S}_\jg$ of non-real hyperbolic roots, we observe that it consists of roots
$\psi=\alpha\eg_1+\beta\eg_2$ with $\alpha$ and $\beta$ distinct real roots of $P_\omega$, hence
$2s_\jg=r(r-1)$.
For the sets $\mathcal{S}_\ig$ and $\mathcal{S}_\kg$, note that $a+\ig b\in\C_\ig=\C$ is a root of $P_\omega$
if and only if $a+\kg b\in\C_\kg$ is a root, which implies that
$2s_\ig=2s_\kg=2s$.
Finally, the expression for $4d$ follows from \eqref{eqroots} together with the relation $n=r+2s$.
\end{proof}

Since $\Lambda(\omega)$ consists of the $n=r+2s$ distinct complex roots of $P_\omega$, we have the factorization
$$
P_\omega(X)=a_n\prod_{\psi\in\Lambda(\omega)}(X-\psi),
$$
where $a_n$ is the leading coefficient of $P_\omega$. The next theorem gives the corresponding formula over the full root set $\Lambda_\B(\omega)$.
We use the identity
$$
\prod_{(i,j)\in\Omega}\!\!\bigl(X-(\mu_{ij}\eg_1+\eta_{ij}\eg_2)\bigr)
=
\Bigl(\prod_{(i,j)\in\Omega}(X-\mu_{ij})\Bigr)\eg_1
+
\Bigl(\prod_{(i,j)\in\Omega}(X-\eta_{ij})\Bigr)\eg_2,
$$
where $\Omega\subset\{1,\dots,p\}\times\{1,\dots,q\}$ and $\mu_{ij},\eta_{ij}\in\C$.

\begin{theo}\label{thefactorization}
Let $\omega$ be a bicomplex algebraic number of degree $n$. Then
$$
P_\omega^{\,n}(X)=a_n^n \prod_{\psi\in\Lambda_\B(\omega)} (X-\psi).
$$
\end{theo}

\begin{proof}
Let us consider the polynomials $P_{\mathcal{S}}$ defined for each subset $\mathcal{S}$ of $\Lambda_\B(\omega)$ by
\begin{eqnarray*}
	P_{\mathcal{S}}(X)	= \displaystyle\left\{\begin{array}{rl} \displaystyle 1,\mbox{~~~~~~~~~~~~~~~~~~~~~~~~~~~~~~~~~~~~}& \mbox{if~} \mathcal{S}=\emptyset\\\\
	 \displaystyle  \prod_{\psi \in \mathcal{S}} (X-\psi), \mbox{~~~~~~~~~~~~~~~~~~~~~~~} & \mbox{otherwise~} 		×
	\end{array}\right.	
\end{eqnarray*}
Thus, \eqref{unionrootset} yields
\begin{equation}\label{facprod}
\prod_{\psi\in \Lambda_\B(\omega)}(X-\psi)
= P_{\mathcal{R}}P_{\mathcal{S}_\ig}P_{\mathcal{S}_\jg}P_{\mathcal{S}_\kg}P_{\mathcal{D}}.
\end{equation}\\
We now describe each polynomial $P_{\mathcal{S}}$.
Without loss of generality, we may assume that $r,s\neq 0,1$; that is, by Theorem~\ref{theroots}, each subset $\mathcal{S}$ in the decomposition \eqref{unionrootset} is nonempty.
In this case, the set $\Lambda(\omega)$ of all $n=r+2s$ distinct complex roots of $P_\omega$ is given by
\begin{equation}\label{setlamda}
\Lambda(\omega)=\{\alpha_1,\ldots,\alpha_r,\beta_1,\ldots,\beta_s,\bar{\beta}_1,\ldots,\bar{\beta}_s\},
\end{equation}
where the $\alpha_k$ are real roots and the $\beta_l,\bar{\beta}_l$ are non-real complex roots, for $k=1,\ldots,r$ and $l=1,\ldots,s$.
Therefore, $P_\omega(X)$ admits the factorization
\begin{equation}\label{factoPw}
P_\omega(X)=a_n\prod_{k=1}^r(X-\alpha_k)\prod_{l=1}^s (X-\beta_l)(X-\bar{\beta}_l).
\end{equation}\\
Case $\mathcal{S}=\mathcal{R}$.
We have
\begin{equation}\label{PR}
P_{\mathcal{R}}=\prod_{k=1}^r(X-\alpha_k).
\end{equation}\\
Case $\mathcal{S}=\mathcal{S}_\ig,\mathcal{S}_\kg$.
We have
\begin{equation}\label{PSi}
P_{\mathcal{S}_{\ig}}=P_{\mathcal{S}_{\kg}}
=\prod_{l=1}^s (X-\beta_l)(X-\bar{\beta}_l).
\end{equation}
Indeed, as indicated previously, a complex number $a+\ig b$ is a root of $P_\omega$ if and only if $a+\kg b$ is a root of $P_\omega$ in $\C_\kg$.\\
Case $\mathcal{S}=\mathcal{S}_\jg$.
We have
$$
\mathcal{S}_{\jg}=\{\alpha_i\eg_1+\alpha_j\eg_2:\ 1\leq i,j\leq r,\ i\neq j\}.
$$
Thus,
\begin{eqnarray*}
\prod_{\psi\in \mathcal{S}_\jg}(X-\psi)
&=&
\prod_{\substack{1\leq i,j\leq r \\ i\neq j}} (X-\alpha_i)\eg_1
+
\prod_{\substack{1\leq i,j\leq r \\ i\neq j}} (X-\alpha_j)\eg_2\\
&=&
\prod_{\substack{1\leq i,j\leq r \\ i\neq j}} (X-\alpha_i).
\end{eqnarray*}
Hence,
\begin{equation}\label{PSj}
P_{\mathcal{S}_\jg}
=\prod_{k=1}^r (X-\alpha_k)^{r-1}.
\end{equation}\\
Case $\mathcal{S}=\mathcal{D}$.
We know that $\mathcal{D}$ is closed under conjugation operations and consists of all bicomplex numbers
$\psi=\lambda\eg_1+\gamma\eg_2$ with $\lambda,\gamma\in\Lambda(\omega)$ such that
$\overline{\psi}^{\ug}\neq\psi$ for $\ug=\ig,\jg,\kg$.
From the idempotent representation of $\psi$, we have
$\overline{\psi}^{\ig}=\gamma\eg_1+\lambda\eg_2$,
$\overline{\psi}^{\jg}=\bar{\lambda}\eg_1+\bar{\gamma}\eg_2$,
and
$\overline{\psi}^{\kg}=\bar{\gamma}\eg_1+\bar{\lambda}\eg_2$.
Thus, from \eqref{setlamda}, we obtain that $\mathcal{D}$ is the disjoint union
$$
\mathcal{D}=\mathcal{D}_1\cup \mathcal{D}_2,
$$
where
\begin{eqnarray*}
\mathcal{D}_1
&=&
\{\alpha_i\eg_1+\beta_j\eg_2,\,
\beta_j\eg_1+\alpha_i\eg_2,\,
\alpha_i\eg_1+\bar{\beta}_j\eg_2,\,
\bar{\beta}_j\eg_1+\alpha_i\eg_2:
1\leq i\leq r,\ 1\leq j\leq s\},\\[1ex]
\mathcal{D}_2
&=&
\{\beta_i\eg_1+\beta_j\eg_2,\,
\bar{\beta}_i\eg_1+\bar{\beta}_j\eg_2,\,
\beta_i\eg_1+\bar{\beta}_j\eg_2,\,
\bar{\beta}_j\eg_1+\beta_i\eg_2:
1\leq i,j\leq s,\ i\neq j\}.
\end{eqnarray*}\\
Thus,
\begin{eqnarray*}
\prod_{\psi\in \mathcal{D}_1}(X-\psi)
&=&
\prod_{\substack{1\leq i\leq r \\ 1\leq j\leq s}}
(X-\alpha_i)^2(X-\beta_j)(X-\bar{\beta}_j)\eg_1\\
&&+
\prod_{\substack{1\leq i\leq r \\ 1\leq j\leq s}}
(X-\alpha_i)^2(X-\beta_j)(X-\bar{\beta}_j)\eg_2\\
&=&
\prod_{\substack{1\leq i\leq r \\ 1\leq j\leq s}}
(X-\alpha_i)^2(X-\beta_j)(X-\bar{\beta}_j).
\end{eqnarray*}
Hence,
$$
P_{\mathcal{D}_1}
=\prod_{k=1}^r (X-\alpha_k)^{2s}
\prod_{l=1}^s \bigl((X-\beta_l)(X-\bar{\beta}_l)\bigr)^r.
$$\\
For $P_{\mathcal{D}_2}$, one has
\begin{eqnarray*}
\prod_{\psi\in \mathcal{D}_2}(X-\psi)
&=&
\prod_{\substack{1\leq i,j\leq s \\ i\neq j}}
\bigl((X-\beta_i)(X-\bar{\beta}_i)\bigr)^2\eg_1\\
&&+
\prod_{\substack{1\leq i,j\leq s \\ i\neq j}}
\bigl((X-\beta_j)(X-\bar{\beta}_j)\bigr)^2\eg_2\\
&=&
\prod_{\substack{1\leq i,j\leq s \\ i\neq j}}
\bigl((X-\beta_i)(X-\bar{\beta}_i)\bigr)^2.
\end{eqnarray*}
Hence,
$$
P_{\mathcal{D}_2}
=\prod_{l=1}^s \bigl((X-\beta_l)(X-\bar{\beta}_l)\bigr)^{2(s-1)}.
$$\\
It follows from the above, using $P_{\mathcal{D}}=P_{\mathcal{D}_1}P_{\mathcal{D}_2}$, that
\begin{equation}\label{PD}
P_{\mathcal{D}}
=\prod_{k=1}^r (X-\alpha_k)^{2s}
\prod_{l=1}^s \bigl((X-\beta_l)(X-\bar{\beta}_l)\bigr)^{r+2(s-1)}.
\end{equation}
Finally, since $n=r+2s$, substituting \eqref{PD}, \eqref{PR}, \eqref{PSi}, and \eqref{PSj} into \eqref{facprod} and using \eqref{factoPw}, we obtain
$$
\prod_{\psi\in \Lambda_\B(\omega)}(X-\psi)
= a_n^{-n}P_\omega^n(X).
$$
\end{proof}

\begin{exa}
\rm
Let $\omega = 1+\ig+\jg+\kg$, which does not lie in any $\C_\ug$. In idempotent form,
$$
\omega = 2\eg_1 + 2\ig\eg_2.
$$
By Theorem~\ref{thpolyminimal},
$$
P_\omega(X)=(X-2)(X^2+4)=X^3-2X^2+4X-8.
$$
The complex root set is $\Lambda(\omega)=\{2,-2\ig,2\ig\}$, and the $3^2$ bicomplex roots are
$$
\Lambda_\B(\omega)=\{\alpha\eg_1+\beta\eg_2:\alpha,\beta\in\Lambda(\omega)\}.
$$
By Theorem~\ref{thefactorization},
$$
\prod_{\psi\in\Lambda_\B(\omega)}(X-\psi)
=
P_\omega^{\,3}(X).
$$
\end{exa}

\begin{exa}
\rm
Let $n\geq 2$ and set $\zeta_n = e^{\frac{2\ig\pi}{n}}$. This is a complex algebraic integer whose minimal polynomial is the $n$th cyclotomic polynomial $\Phi_n \in \Z[X]$ of degree $\varphi(n)$, given by
$$
\Phi_n(X) = \prod_{\substack{1 \leq m \leq n \\ (m,n)=1}} (X - \zeta_n^{m}).
$$
Since 
$$
\Lambda(\zeta_n)=\{\zeta_n^{m} : 1 \leq m \leq n,\ (m,n)=1\}
$$
is the set of the $\varphi(n)$ distinct complex roots of $\Phi_n$, the $\varphi(n)^2$ distinct bicomplex roots of $\Phi_n$ are precisely
$$
\Lambda_\B(\zeta_n)
=
\{\zeta_n^{m}\eg_1 + \zeta_n^{m'}\eg_2 :
1 \leq m,m' \leq n,\ (m,n)=(m',n)=1\}.
$$
Therefore, by Theorem~\ref{thefactorization}, we obtain
$$
\prod_{\psi \in \Lambda_\B(\zeta_n)} (X - \psi)
=
\Phi_n^{\varphi(n)}(X) \in \Z[X].
$$
\end{exa}

\section{Finite bicomplex extensions of $\Q$}

In this section, we consider finite bicomplex extensions $L$ of $\Q$; that is, finite unital $\Q$-algebras $L \subset \B$. Our aim is to study their rings of integers $\mathcal{O}_L$. Recall that $\mathcal{O}_L$ consists of all elements of $L$ (hence all bicomplex algebraic numbers) which satisfy monic polynomials over $\Z$. For basic concepts from algebraic number theory and abstract algebra, we refer the reader to \cite{Murty,Milne,L}.

The next theorem plays a central role in understanding the basic properties of $\mathcal{O}_L$.

\begin{theo} \label{decompalg}
Let $L$ be a finite bicomplex extension of $\Q$. Then $L = \Q[\omega]$ for some $\omega \in L$. Moreover, $L$ is a field extension of $\Q$ if and only if $\omega \in \C_\ug$ for some $\ug \in \{\ig, \kg\}$. Otherwise,
$$
L = K_1 \eg_1 + K_2 \eg_2,
$$
where $K_1$ and $K_2$ are number fields.
\end{theo}

\begin{proof}
Let $L$ be a finite bicomplex extension of $\Q$. Since $L$ is bicomplex algebraic, it follows that $L$ is a field extension if and only if there exist $\ug \in \{\ig, \kg\}$ and a field $K \subset \C_\ug$ such that $L = K$, that is, $L = \Q[\omega]$ for some $\omega \in \C_\ug$. Otherwise, $L$ decomposes uniquely as
$$
L = K_1 \eg_1 + K_2 \eg_2,
$$
where $K_k = \mathcal{P}_k(L) \subset \C$ for $k = 1,2$. Since $L$ is a bicomplex algebraic extension—i.e., a subring of $\mathcal{A}_\B$—Theorem~\ref{bicomalgebr} implies that $K_1$ and $K_2$ are subrings of $\hat{\Q}$, and hence number fields.
Because $L$ is not contained in $\C_\ug$ for $\ug \in \{\ig, \kg\}$, we may choose $\gamma_1 \in K_1$ and $\gamma_2 \in K_2$ with $\gamma_2 \neq \gamma_1$ and $\gamma_1 \neq \bar{\gamma}_2$, such that $K_1 = \Q(\gamma_1)$ and $K_2 = \Q(\gamma_2)$. Setting
$$
\omega = \gamma_1 \eg_1 + \gamma_2 \eg_2,
$$
we have $\omega \notin \C_\ug$ for $\ug \in \{\ig, \kg\}$. By Theorem~\ref{thpolyminimal}, the minimal polynomial of $\omega$ satisfies
$$
P_\omega = P_{\gamma_1} P_{\gamma_2},
$$
where $P_{\gamma_1}$ and $P_{\gamma_2}$ are irreducible and relatively prime. By the Chinese Remainder Theorem we obtain
$$
\Q[\omega] \simeq \Q[X]/(P_{\gamma_1}) \times \Q[X]/(P_{\gamma_2}) \simeq K_1 \times K_2.
$$
Since $L \simeq K_1 \times K_2$ via the ring isomorphism
$
\psi \longmapsto (\mathcal{P}_1(\psi), \mathcal{P}_2(\psi)),
$
it follows that
$$
\Q[\omega] \simeq L.
$$
It remains to prove that this isomorphism is in fact an equality. It suffices to show that $\Q[\omega] \subset L$. Let $\psi \in \Q[\omega]$. Then $\psi = P(\omega)$ for some polynomial $P \in \Q[X]$. Using~\eqref{decompoly}, we have
$$
P(\omega) = P(\gamma_1)\eg_1 + P(\gamma_2)\eg_2,
$$
with $P(\gamma_1) \in K_1$ and $P(\gamma_2) \in K_2$, and therefore $\psi \in K_1 \eg_1 + K_2 \eg_2 = L$, as required.\\
Thus every finite bicomplex extension of $\Q$ is generated by an element $\omega = \gamma_1 \eg_1 + \gamma_2 \eg_2 \in L$, and $L$ is a field if and only if $\omega \in \C_\ug$ for some $\ug \in \{\ig, \kg\}$. Otherwise,
$$
L = K_1 \eg_1 + K_2 \eg_2,
$$
where $K_k = \Q(\gamma_k)$ for $k = 1,2$.
\end{proof}

As shown in Theorem~\ref{decompalg}, a bicomplex extension generated by $\omega \in \C_\ug$ with $\ug = \ig$ or $\ug = \kg$ is simply a number field. However, throughout the remainder of this section we shall be interested in bicomplex extensions $L$ of degree $n$ generated by elements $\omega \notin \C_\ug$ for $\ug \in \{\ig, \kg\}$; that is,
$$
L = K_1 \eg_1 + K_2 \eg_2,
$$
where $K_1$ and $K_2$ are number fields.


\subsection{Rings of integers in bicomplex extensions}

\begin{theo}\label{decompintgers}
The ring of integers $\mathcal{O}_L$ of $L$ decomposes uniquely as
$$
\mathcal{O}_L = \mathcal{O}_{K_1}\eg_1 + \mathcal{O}_{K_2}\eg_2.
$$
Moreover, $\mathcal{O}_L$ is a free $\Z$-module of rank $n$ and contains $\Z$.
\end{theo}

\begin{proof}
Let $\psi = \alpha_1 \eg_1 + \alpha_2 \eg_2 \in L$, where $\alpha_k \in K_k$ for $k = 1,2$.  
By Theorem~\ref{thpolyminimal}, either $P_\psi = P_{\alpha_1} = P_{\alpha_2}$ or
$
P_\psi = P_{\alpha_1} P_{\alpha_2}.
$
Since the leading coefficient of a minimal polynomial is assumed to be a positive integer, it follows that $P_\psi$ is monic if and only if both $P_{\alpha_1}$ and $P_{\alpha_2}$ are monic.  
Thus,
$$
\psi \in \mathcal{O}_L \quad \Longleftrightarrow \quad \alpha_k \in \mathcal{O}_{K_k} \text{ for } k = 1,2,
$$
and therefore
$$
\mathcal{O}_L = \mathcal{O}_{K_1}\eg_1 + \mathcal{O}_{K_2}\eg_2.
$$
To prove that $\mathcal{O}_L$ is a free $\Z$-module of rank $n$, choose
$
\{\eg_1 f_1, \ldots, \eg_1 f_r,\, \eg_2 g_1, \ldots, \eg_2 g_s\}
$
as a $\Z$-basis of $\mathcal{O}_L$, where  
$\{f_1,\ldots,f_r\}$ and $\{g_1,\ldots,g_s\}$ are $\Z$-bases of $\mathcal{O}_{K_1}$ and $\mathcal{O}_{K_2}$, respectively.  
Finally, since any integer $m \in \Z$ can be written as
$
m = m\eg_1 + m\eg_2,
$
we conclude that $\mathcal{O}_L$ contains $\Z$.
\end{proof}
\subsection{Discriminant of the bicomplex extension $L/\Q$} 

Recall that if $R$ is a ring containing $\Z$ and free as a $\Z$-module of rank $n$, then  
\begin{equation}\label{discrdefi}
D(\beta_1, \ldots, \beta_n):=\det\!\left( \mathrm{Tr}_{R/\Z}( \beta_i \beta_j ) \right),
\end{equation}
is a well–defined integer, independent of the chosen $\Z$-basis 
$\{\beta_1,\ldots,\beta_n\}$. It is denoted by $\mathrm{disc}(R/\Z)$ and is called the \emph{discriminant} of $R$ over $\Z$.
When $K$ is a number field of degree $n$, its ring of integers $\mathcal{O}_K$ is a free $\Z$-module of rank $n$, and  
$\mathrm{disc}(\mathcal{O}_K/\Z)$ is usually referred to as the discriminant of the extension $K/\Q$.

For the bicomplex extensions $L = K_1 \eg_1 + K_2 \eg_2$ of degree $n$, Theorem~\ref{decompintgers} shows that 
$\mathcal{O}_L$ is a free $\Z$-module of rank $n$ containing $\Z$. Hence 
$\mathrm{disc}(\mathcal{O}_L/\Z)$ is well defined, and for any choice of $\Z$-basis 
$\{\beta_1,\ldots,\beta_n\}$ of $\mathcal{O}_L$ it may be computed from \eqref{discrdefi}.  
The following proposition provides a convenient formula.

\begin{prop}\label{disriminant}
The discriminant of $L$ is given by
$$
\mathrm{disc}(\mathcal{O}_L/\Z)
   = \mathrm{disc}(\mathcal{O}_{K_1}/\Z)\;
     \mathrm{disc}(\mathcal{O}_{K_2}/\Z).
$$
\end{prop}

\begin{proof}
By Theorem~\ref{decompintgers}, the ring $\mathcal{O}_L$ is a free $\Z$-module and we have the isomorphism
\begin{equation}\label{ringintergiso}
\mathcal{O}_L \simeq \mathcal{O}_{K_1} \times \mathcal{O}_{K_2},
\end{equation}
via the map $\varphi(\gamma_1 \eg_1 + \gamma_2 \eg_2) = (\gamma_1, \gamma_2)$.  
The claimed identity then follows from Lemma~3.37 in \cite[Chapter~3]{Milne}, applied to the pair 
$(\mathcal{O}_{K_1}, \mathcal{O}_{K_2})$ of free $\Z$-modules.
\end{proof}
      \subsection{Units in $\mathcal{O}_L$}
      
      We now describe the unit group of $\mathcal{O}_L$ and characterize the cases in which it is finite.
      
      \begin{prop}\label{unitgroup}
      The unit group of $\mathcal{O}_L$ decomposes as
     $$
      \mathcal{O}_L^{\times}
         = \mathcal{O}_{K_1}^{\times}\eg_1 \;+\; \mathcal{O}_{K_2}^{\times}\eg_2 .
     $$
      Moreover, $\mathcal{O}_L^{\times}$ is finite if and only if $L$ belongs to one of the following three classes of bicomplex extensions:
      \begin{itemize}
          \item[$(C_1)$] the quadratic hyperbolic extension  
          $$
          L=\Q\eg_1 + \Q\eg_2 = \Q[\jg];
          $$
          \item[$(C_2)$] cubic extensions of the form  
         $$
          L = \Q(\ig\sqrt{d})\,\eg_1 + \Q\,\eg_2 
          \;\backsimeq\; \Q\,\eg_1 + \Q(\ig\sqrt{d})\,\eg_2;
          $$
          \item[$(C_3)$] quadratic extensions  
          $$
          L=\Q(\ig\sqrt{d})\,\eg_1 + \Q(\ig\sqrt{d'})\,\eg_2,
          $$
      \end{itemize}
      where $d,d'$ are positive squarefree integers.
      \end{prop}
      
      \begin{proof}
      From the ring isomorphism \eqref{ringintergiso}, we have
      $$
      \mathcal{O}_L^{\times}
         \backsimeq \mathcal{O}_{K_1}^{\times} \times \mathcal{O}_{K_2}^{\times}
         \backsimeq \mathcal{O}_{K_1}^{\times}\eg_1 + \mathcal{O}_{K_2}^{\times}\eg_2.
      $$
      Thus, to obtain the stated decomposition of $\mathcal{O}_L^{\times}$, it suffices to show the inclusion  
      $$
      \mathcal{O}_L^{\times} 
         \subset \mathcal{O}_{K_1}^{\times}\eg_1 + \mathcal{O}_{K_2}^{\times}\eg_2.
     $$
      Indeed, let $\psi=\alpha_1\eg_1 + \alpha_2\eg_2 \in \mathcal{O}_L$ be a unit. Then there exists 
      $\omega = \beta_1\eg_1 + \beta_2\eg_2 \in \mathcal{O}_L$ such that 
      $\psi\omega = 1$, which implies
      $
      \alpha_1\beta_1 = \alpha_2\beta_2 = 1.
      $
      By Theorem~\ref{decompintgers}, we have $\alpha_k, \beta_k \in \mathcal{O}_{K_k}$ for $k=1,2$, hence  
      $\alpha_k \in \mathcal{O}_{K_k}^{\times}$. This proves  
      $\psi \in \mathcal{O}_{K_1}^{\times}\eg_1 + \mathcal{O}_{K_2}^{\times}\eg_2$, and therefore
      $$
      \mathcal{O}_L^{\times}
         = \mathcal{O}_{K_1}^{\times}\eg_1 + \mathcal{O}_{K_2}^{\times}\eg_2.
      $$
      From this decomposition, $\mathcal{O}_L^{\times}$ is finite if and only if 
      $\mathcal{O}_{K_1}^{\times}$ and $\mathcal{O}_{K_2}^{\times}$ are finite.  
      The desired classification of bicomplex extensions $L=K_1\eg_1 + K_2\eg_2$ with finite unit group 
      then follows immediately from this observation and from the classical fact that  
      the unit group $\mathcal{O}_K^{\times}$ of a number field $K$ is finite if and only if  
      $K=\Q$ or $K=\Q(\ig\sqrt{d})$ with $d>0$ squarefree (see \cite[Chapter~8, Exercise~8.1.6]{Murty}).
      \end{proof}
     
     
    \subsection{Ideals of $\mathcal{O}_L$ and Dedekind zeta function}
     In this section we describe the structure of ideals and prime ideals in the ring of integers $\mathcal{O}_L$ of a bicomplex extension $L=\Q[\omega]= K_1\eg_1+K_2\eg_2$, where $\omega\notin\C_\ug$ for $\ug\in\{\ig,\kg\}$. 
    This decomposition allows us, in a natural way, to introduce the arithmetic zeta function associated with $L$ and to express it in terms of Dirichlet series arising from the components $K_1$ and $K_2$. By Theorem \ref{decompintgers}, one has the decomposition
    $$
    \mathcal{O}_L=(\eg_1)+(\eg_2),\qquad 
    (\eg_1)\cap(\eg_2)=(0).
   $$
    The ideals $(\eg_k)=\eg_k\mathcal{O}_L=\eg_k\mathcal{O}_{K_k}$ will be called \emph{degenerate ideals} of $\mathcal{O}_L$.
    
    \begin{prop}\label{ideal}
    Every ideal $\fraka$ of $\mathcal{O}_L$ is of the form
    $$
    \fraka=\fraka_1\eg_1+\fraka_2\eg_2,
    $$
    where $\fraka_k$ is an ideal of $\mathcal{O}_{K_k}$ for $k=1,2$.
    \end{prop}
    
    \begin{proof}
    From the decomposition $\mathcal{O}_L=\mathcal{O}_{K_1}\eg_1+\mathcal{O}_{K_2}\eg_2$ given in Theorem \ref{decompintgers}, the maps
    \begin{equation}\label{phikmaps}
    \varphi_k:\mathcal{O}_L\longrightarrow\mathcal{O}_{K_k},\qquad 
    \varphi_k(\gamma_1\eg_1+\gamma_2\eg_2)=\gamma_k,
    \end{equation}
    are surjective ring homomorphisms.  
    Thus every ideal $\fraka\subseteq\mathcal{O}_L$ satisfies
    $$
    \fraka=\varphi_1(\fraka)\eg_1+\varphi_2(\fraka)\eg_2,
    $$
    where $\fraka_k=\varphi_k(\fraka)$ is an ideal of $\mathcal{O}_{K_k}$ for $k=1,2$.
    \end{proof}
    
    \begin{prop}\label{generideal}
    Every ideal of $\mathcal{O}_L$ is finitely generated.
    \end{prop}
    
    \begin{proof}
    Let $\fraka$ be an ideal of $\mathcal{O}_L$.  
    By Proposition \ref{ideal},
    $$
    \fraka=\fraka_1\eg_1+\fraka_2\eg_2,\qquad 
    \fraka_k\subseteq\mathcal{O}_{K_k}.
    $$
    Since ideals in $\mathcal{O}_{K_k}$ are finitely generated, write
    $$
    \fraka_1=(a_1,\dots,a_r),\qquad 
    \fraka_2=(b_1,\dots,b_s).
    $$
    Assume $r\le s$ and define $c_k\in\mathcal{O}_L$ by
    $$
    c_k=a'_k\eg_1+b_k\eg_2,\qquad 
    a'_k=
    \begin{cases}
    a_k,& k\le r,\\[2pt]
    0,& k>r.
    \end{cases}
    $$
    Then $\fraka=(c_1,\dots,c_s)$.
    \end{proof}
    
    We now characterize prime ideals of $\mathcal{O}_L$.  
    Recall that an ideal $\fraka$ in a commutative ring $R$ is prime if $ab\in\fraka$ implies $a\in\fraka$ or $b\in\fraka$, equivalently if $R/\fraka$ is an integral domain.  In particular, the zero ideal $(0)$ is prime if and only if $R$ itself is an integral domain.\\  
    In the present setting, the zero ideal of $\mathcal{O}_L$ is prime if and only if $L=\Q[\omega]$ with $\omega \in \C_{\ug}$ for some $\ug \in \{\ig,\kg\}$. Nevertheless, as we shall see below, the degenerate ideals $(\eg_1)$ and $(\eg_2)$ are prime ideals of $\mathcal{O}_L$.
    
    \begin{prop}\label{primeideal}
    The prime ideals of $\mathcal{O}_L$ are exactly the ideals of the form
    $$
    \frakp_1\eg_1+\mathcal{O}_{K_2}\eg_2,
    \qquad
    \mathcal{O}_{K_1}\eg_1+\frakp_2\eg_2,
    $$
    where $\frakp_k$ is a prime ideal of $\mathcal{O}_{K_k}$ for $k=1,2$.  
    In particular, the degenerate ideals $(\eg_k)=\eg_k\mathcal{O}_{K_k}$ are prime.
    \end{prop}
    
    \begin{proof}
    Let $\fraka$ be an ideal of $\mathcal{O}_L$.  
    By Proposition \ref{ideal}, 
    $$
    \fraka=\fraka_1\eg_1+\fraka_2\eg_2,\qquad 
    \fraka_k=\varphi_k(\fraka)\subseteq\mathcal{O}_{K_k},
    $$
    where $\varphi_k$ is defined in \eqref{phikmaps}.  
    Define
    $$
    \bar{\varphi} : \mathcal{O}_L/\fraka \longrightarrow
    \mathcal{O}_{K_1}/\fraka_1 \times \mathcal{O}_{K_2}/\fraka_2,
    \qquad
    \bar{\varphi}[\varsigma]
    =\big([\varphi_1(\varsigma)]_1,\,[\varphi_2(\varsigma)]_2\big).
    $$
    This yields the ring isomorphism
    \begin{equation}\label{eqisomideal}
    \mathcal{O}_L/\fraka
    \simeq
    \mathcal{O}_{K_1}/\fraka_1
    \times
    \mathcal{O}_{K_2}/\fraka_2.
    \end{equation}
    The product on the right is an integral domain precisely when one factor is zero and the other is an integral domain.  
    Thus $\fraka$ is prime if and only if one of the $\fraka_k$ equals $\mathcal{O}_{K_k}$ and the other is a prime ideal of $\mathcal{O}_{K_k}$.  
    Taking $\frakp_k=(0)$, which is prime in $\mathcal{O}_{K_k}$, shows that $( \eg_k)= \eg_k \mathcal{O}_L= \eg_k \mathcal{O}_{K_k}$ is prime in $\mathcal{O}_L$ for $k=1,2$. 
    \end{proof}
    
    Recall that if $K$ is a number field, every non-zero ideal $\fraka\subseteq\mathcal{O}_K$ has finite index, and its norm is
    $$
    \mathrm{N}_K(\fraka):=\big|\mathcal{O}_K/\fraka\big|.
    $$
    The Dedekind zeta function of $K$ is the Dirichlet series
    $$
    \zeta_K(s)=\sum_{(0)\neq\fraka\subseteq\mathcal{O}_K}
    \frac{1}{\mathrm{N}_K(\fraka)^s}
    =\sum_{n\ge1}\frac{a_K(n)}{n^s},\qquad s \in \mathbb{C},\;
    \Rj(s)>1,
    $$
    where $ a_K(n)$ is the number of ideals in $\mathcal{O}_K $ of norm $n$.
    Let $L=K_1\eg_1+K_2\eg_2$ be a bicomplex extension.  
    Then, by \eqref{eqisomideal}, if $\fraka=\fraka_1\eg_1+\fraka_2\eg_2$ is non-zero and non-degenerate, the $\fraka_k$ are non-zero, and
    $$
    \big|\mathcal{O}_L/\fraka\big|
    =
    \big|\mathcal{O}_{K_1}/\fraka_1\big|\,
    \big|\mathcal{O}_{K_2}/\fraka_2\big|
    =\mathrm{N}_{K_1}(\fraka_1)\,\mathrm{N}_{K_2}(\fraka_2),
    $$
    which defines the norm $\mathrm{N}_L(\fraka)$ and yields
    \begin{equation}\label{eqnormideal}
    \mathrm{N}_L(\fraka)
    =
    \mathrm{N}_{K_1}(\fraka_1)\,
    \mathrm{N}_{K_2}(\fraka_2).
    \end{equation}
    
    \begin{defi}\label{zeta L}
    The zeta function of the bicomplex extension $L$ is defined for $s \in \mathbb{C}$ and
    $\Rj(s)>1$ by
    $$
    \zeta_L(s)
    :=
    \sum_{\substack{(0)\neq\fraka\subseteq\mathcal{O}_L\\
    \fraka\neq(\eg_1),(\eg_2)}}
    \frac{1}{\mathrm{N}_L(\fraka)^s}.
    $$
    \end{defi}
    
    \begin{prop}\label{eqzeta}
    For every bicomplex extension $L=K_1\eg_1+K_2\eg_2$,
    $$
    \zeta_L(s)=\zeta_{K_1}(s)\,\zeta_{K_2}(s).
    $$
    \end{prop}
    
    \begin{proof}
    Using Definition \ref{zeta L} and \eqref{eqnormideal},
    \begin{align*}
    \zeta_L(s)
    &=
    \sum_{\substack{(0)\neq\fraka\subseteq\mathcal{O}_L\\
    \fraka\neq(\eg_1),(\eg_2)}}
    \frac{1}{\mathrm{N}_L(\fraka)^s}  \\[4pt]
    &=\left(
    \sum_{(0)\neq\fraka_1\subseteq\mathcal{O}_{K_1}}
    \frac{1}{\mathrm{N}_{K_1}(\fraka_1)^s}
    \right)
    \left(
    \sum_{(0)\neq\fraka_2\subseteq\mathcal{O}_{K_2}}
    \frac{1}{\mathrm{N}_{K_2}(\fraka_2)^s}
    \right)  \\[4pt]
    &=\zeta_{K_1}(s)\,\zeta_{K_2}(s).
    \end{align*}
    \end{proof}
    
    Therefore, the generating function $a_L(n)$ of $\zeta_L$ satisfies
    \begin{equation}\label{eqgenerating}
    a_L(n)=(a_{K_1}\ast a_{K_2})(n),\qquad n\ge1,
    \end{equation}
    where $f\ast g$ denotes the Dirichlet convolution:
    $$
    (f\ast g)(n)
    :=
    \sum_{pq=n} f(p)g(q)
    =
    \sum_{d\mid n} f(d)\,g(n/d),\qquad n\ge1.
    $$
     
    \subsection{Prime and irreducible elements in $\mathcal{O}_L$}
    
    Recall that an element $\pi$ in a commutative ring $R$ is said to be \emph{prime} (resp.\ \emph{irreducible}) if $\pi$ is non-zero and non-unit, and if the ideal $(\pi)$ is prime (resp.\ if $\pi = ab$ implies that either $a$ or $b$ is a unit).  
    If $R$ is a principal ideal domain, then irreducible elements are prime and $R$ is a unique factorization domain. In this case, every non-zero and non-unit element $a \in R$ admits a unique factorization of the form
    \begin{equation}\label{eqfacto}
        a = u \pi_1^{\alpha_1} \cdots \pi_k^{\alpha_k},
    \end{equation}
    where $u$ is a unit, the $\pi_j$ are irreducible elements of $R$, and $\alpha_j \in \mathbb{Z}_{>0}$ for all $j=1,\dots,k$.\\
    We show next that this property extends naturally to the ring of integers $\mathcal{O}_L$.  
    Throughout, we say that \emph{unique factorization holds in $\mathcal{O}_L$} if every $\upsilon \in \mathcal{O}_L$ with $N(\upsilon)\neq 0$ and $\upsilon \notin \mathcal{O}_L^\times$ admits a unique factorization of the form~\eqref{eqfacto}.  
    To this end, we first establish a precise description of prime elements in $\mathcal{O}_L$.
    
    \begin{prop}\label{primes}
    Up to multiplication by a unit, the prime elements in $\mathcal{O}_L$ are
    \begin{equation}\label{eqprimes}
               \eg_1,\qquad \eg_2,\qquad 
               \pi_1 \eg_1 + \eg_2,\qquad 
               \eg_1 + \pi_2 \eg_2,
           \end{equation}
    where, for $k=1,2$, the element $\pi_k$ is a prime element of $\mathcal{O}_{K_k}$.\\
    In particular, if $\mathcal{O}_L$ is a principal ring, then the irreducible elements are primes with non-zero norm, i.e. the primes $\pi$ of the form , up to units,
  $$
            \pi_1 \eg_1 + \eg_2\qquad or \qquad
           \eg_1 + \pi_2 \eg_2.
       $$
    \end{prop}
    
    \begin{proof}
    Let $\pi \in \mathcal{O}_L$ be non-zero and non-unit. Then $\pi$ is prime if and only if $(\pi)$ is a prime ideal, i.e.\ $\mathcal{O}_L/(\pi)$ is an integral domain.  
    Using Proposition~\ref{primeideal}, together with the fact that the zero ideal is prime in an integral domain, we obtain that $\pi$ is, up to unit, of one of the following forms:
    $$
        \eg_1,\qquad \eg_2,\qquad 
        \pi_1 \eg_1 + \eg_2,\qquad 
        \eg_1 + \pi_2 \eg_2,
    $$
    where $\pi_k$ is a prime element of $\mathcal{O}_{K_k}$ for $k=1,2$.    
    Now assume that $\mathcal{O}_L$ is a principal ring.  
    If $\pi$ is irreducible, then $(\pi)$ is maximal, hence $\mathcal{O}_L/(\pi)$ is a field, in particular an integral domain.  
    Thus $\pi$ is prime and therefore, up to unit, of one of the forms in~\eqref{eqprimes}.  
    Since $\eg_k$ is not irreducible (because $\eg_k^2=\eg_k$), it remains to prove that the other two types are irreducible. Consider $\pi = \pi_1 \eg_1 + \eg_2$ and assume $\pi = ab$ with $a,b \in \mathcal{O}_L$.  
    By Theorem~\ref{decompintgers}, we may write
    $$
        a = a_1\eg_1 + a_2\eg_2 ,\qquad
        b = b_1\eg_1 + b_2\eg_2 ,
    $$
    with $a_k, b_k \in \mathcal{O}_{K_k}$.  
    Then
    $$
       \pi = \pi_1\eg_1 + \eg_2 = a_1 b_1 \eg_1 + a_2 b_2 \eg_2,
    $$
    so that $\pi_1 = a_1 b_1$ and $a_2 b_2 = 1$.  
    Since $\mathcal{O}_{K_1}$ and $\mathcal{O}_{K_2}$ are principal rings by Proposition~\ref{generideal}, the prime $\pi_1$ is irreducible in $\mathcal{O}_{K_1}$.  
    Thus either $a_1$ or $b_1$ is a unit of $\mathcal{O}_{K_1}$, and since $a_2b_2=1$, both $a_2$ and $b_2$ are units of $\mathcal{O}_{K_2}$.  
    Hence either $a$ or $b$ is a unit in $\mathcal{O}_L$, proving that $\pi$ is irreducible.  
    The same reasoning applies to $\eg_1 + \pi_2 \eg_2$.  
    This completes the proof.
    \end{proof}
    
    \begin{rk}\label{remprime}\rm
    Proposition~\ref{primes} shows that, for any number field $K$, every prime $\pi \in \mathcal{O}_K$ becomes \emph{semiprime} in the ring of integers $\mathcal{O}_L$ of the bicomplex extension $L = K\eg_1 + K\eg_2$.  
    Indeed,
    $$
       \pi = (\pi \eg_1 + \eg_2)(\eg_1 + \pi \eg_2),
    $$
    and each factor is prime in $\mathcal{O}_L$.
    \end{rk}
    
    \begin{prop}\label{unique factorization}
    Unique factorization holds in $\mathcal{O}_L$ whenever $\mathcal{O}_L$ is a principal ring.
    \end{prop}
    
    \begin{proof}
    Assume that $\mathcal{O}_L$ is a principal ring; by Proposition~\ref{generideal} this is equivalent to requiring that each $\mathcal{O}_{K_k}$ is a principal ring.  
    Let $\upsilon = \upsilon_1 \eg_1 + \upsilon_2 \eg_2 \in \mathcal{O}_L$ with $N(\upsilon)\neq 0$ and $\upsilon$ not a unit of  $\mathcal{O}_L$.
    Suppose first that no $\upsilon_k$ is a unit in $\mathcal{O}_{K_k}$.  
    Since each $\mathcal{O}_{K_k}$ is unique factorization domain, we may write
    $$
        \upsilon_k = \mu_k \pi_{1,k}^{\alpha_{1,k}} \cdots \pi_{s_k,k}^{\alpha_{s_k,k}},
    $$
    where $\mu_k \in \mathcal{O}_{K_k}^\times$, the $\pi_{l_k,k}$ are irreducible in $\mathcal{O}_{K_k}$, and the $\alpha_{l_k,k}\in\mathbb{Z}_{>0}$.  
    Setting $\mu = \mu_1\eg_1 + \mu_2\eg_2$, we obtain
    \begin{align*}
        \upsilon
        &= \mu \left(\pi_{1,1}^{\alpha_{1,1}} \cdots \pi_{s_1,1}^{\alpha_{s_1,1}} \eg_1 + \eg_2\right)
           \left(\eg_1 + \pi_{1,2}^{\alpha_{1,2}} \cdots \pi_{s_2,2}^{\alpha_{s_2,2}} \eg_2\right) \\
        &= \mu
           \prod_{l_1=1}^{s_1}\left(\pi_{l_1,1}\eg_1 + \eg_2\right)^{\alpha_{l_1,1}}
           \prod_{l_2=1}^{s_2}\left(\eg_1 + \pi_{l_2,2}\eg_2\right)^{\alpha_{l_2,2}} .
    \end{align*}\\
    If one of the $\upsilon_k$ is a unit, say $\upsilon_2=\mu_2\in\mathcal{O}_{K_2}^\times$, then we obtain the simplified factorization
    $$
        \upsilon
        = \mu
          \prod_{l_1=1}^{s_1}\left(\pi_{l_1,1}\eg_1 + \eg_2\right)^{\alpha_{l_1,1}} .
    $$
    By Proposition~\ref{primes}, each $\pi_{l_1,1}\eg_1 + \eg_2$ and each $\eg_1 + \pi_{l_2,2}\eg_2$ is irreducible in $\mathcal{O}_L$.  
    Thus the factorization is unique.  
    \end{proof}
     
   	    \section{Examples of bicomplex extensions}
   	    In this section, we apply the main results obtained above to describe some properties of specific examples of quadratic and quartic bicomplex extensions.
   	    
   	    \subsection{Bicomplex quadratic extensions}
   	    By Theorem~\ref{decompalg}, bicomplex quadratic extensions $L=\Q[\omega]$ are precisely copies of quadratic fields $\Q(\ug\sqrt{d})$, where $d$ is a positive squarefree integer, whenever $\omega\in\C_\ug$ for some $\ug\in\{\ig,\kg\}$. Otherwise, one has
   	    $$
   	    L=K_1\eg_1+K_2\eg_2,
   	    $$
   	    with $K_1,K_2$ number fields of degree one. In this case, $L$ coincides with the hyperbolic extension $\Qh$, given by
   	    \begin{equation}\label{hyprational}
   	    \Qh=\Q\eg_1+\Q\eg_2=\Q[\jg].
   	    \end{equation}
   	    
   	    The ring $\Qh$ is called the ring of hyperbolic rational numbers. Its ring of integers is the free $\Z$-module
   	    \begin{equation}\label{hypintegers}
   	    \ZB_h=\Z\eg_1+\Z\eg_2=\Z[\eg],\qquad \eg\in\{\eg_1,\eg_2\}.
   	    \end{equation}
   	    The ring $\ZB_h$, called the ring of hyperbolic integers, was studied in \cite{GK3} as the unique (up to ring and order isomorphism) ring of quadratic integers extending fundamental divisibility and order properties of the ordinary integers $\Z$. In particular, $\ZB_h$ is an Archimedean $f$-ring, admits a division algorithm, is a principal ring, and satisfies the unique factorization property. The positive irreducible elements of $\ZB_h$, called \emph{hyperbolic prime numbers}, are precisely of the forms
   	    $$
   	    p^{\eg_1}=p\eg_1+\eg_2
   	    \quad\text{and}\quad
   	    p^{\eg_2}=\eg_1+p\eg_2,
   	    \qquad p \ \text{prime}.
   	    $$
   	    It is worth noting the remarkable fact that a prime number $p$, when viewed as a hyperbolic integer, becomes semiprime, since
   	    $$
   	    p=p^{\eg_1}p^{\eg_2}.
   	    $$
   	     As in the ring of integers $\Z$, every irreducible element $\pi\in\ZB_h$ can be written uniquely as $\pi=\varepsilon\upsilon$, where $\varepsilon$ is a unit and $\upsilon$ is a hyperbolic prime. The group of units of $\ZB_h$ is the Klein group
   	    $$
   	    \ZB_h^\times=\{\pm1,\pm\jg\}\simeq\Z/2\Z\times\Z/2\Z,
   	    $$
   	    which plays the role of the group of signs in the hyperbolic numbers $\D$, analogously to $\{\pm1\}$ in the real numbers $\R$.
   	    
   	    We emphasize that the algebraic and arithmetic properties of $\ZB_h$ mentioned above, first established in \cite{GK3}, follow naturally from our general study of bicomplex extensions by taking $L=\Q[\jg]$. As an additional application, one can compute the discriminant of the extension $\Qh$. By Proposition~\ref{disriminant}, we obtain
   	    $$\mathrm{disc} (\ZB_h/\Z)= \mathrm{disc}^2 ( \Z/\Z) =1.$$
   	    Moreover, in view of Proposition~\ref{eqzeta} and the decomposition \eqref{hyprational}, the zeta function of $\Qh$ is given by
   	    $$
   	    \zeta_{\Qh}(s)=\zeta^2(s),
   	    $$
   	    where $\zeta(s)=\sum_{n\ge1}n^{-s}$ is the Riemann zeta function. Since for each integer $n\ge1$ there exists exactly one ideal of $\Z$ with norm $n$, it follows from \eqref{eqgenerating} that the number $a_{\Qh}(n)$ of ideals $\fraka\subseteq\ZB_h$ with norm $n$ is given by
   	    $$
   	    a_{\Qh}(n)=\sum_{d\mid n}1=d(n).
   	    $$
 
 We conclude this section by highlighting some major differences between the classical imaginary quadratic extension $\Q(\ig)$ and the quadratic hyperbolic extension $\Qh=\Q[\jg]$, which will be discussed below.
 
 \begin{rk}\rm{}
The quadratic number field $\Q(\ig)$ and, consequently, its ring of integers $\Z[\ig]$, the so-called Gaussian integers, have played a central role in the development of several fundamental results in number theory. Although $\Z[\ig]$ shares many similarities with the ordinary integers $\Z$ and has strongly influenced the development of algebraic number theory, it is not the natural setting for generalizing arithmetic notions that are inherently linked to the order structure of $\Z$. In the language of lattice-ordered rings (see, for instance, \cite{GK3}), $\Z$ is an Archimedean $f$-ring. As observed in Subsection~5.1, such a structure extends uniquely to a quadratic ring of integers, namely the ring $\ZB_h$ of integers of the hyperbolic extension $\Qh=\Q[\jg]$. This constitutes one of the main structural differences between these two rings.
 
 Another noteworthy distinction concerns radix representations of complex and hyperbolic numbers. As shown in \cite{KaSa}, for a Gaussian integer $q=a\pm i$ with $a\leq -1$, every complex number $z$ admits an expansion of the form
 $$
 z=\sum_{-\infty}^{k} r_i q^i,\qquad r_i\in\mathfrak{N}_q=\{0,1,\ldots,a^2\}.
 $$
 Here, the Gaussian integer $q$ is called the base of the expansion, while the elements of $\mathfrak{N}_q$ are referred to as its digits.\\  
 In the hyperbolic setting, it is proved in \cite{Kdig} that, up to conjugation, there exist exactly two types of bases $q\in\ZB_h$ allowing such radix representations for all hyperbolic numbers. These are hyperbolic integers of the form
 $$
 q=a\eg_1+(1-a)\eg_2\notin\Z[\jg],
 $$
 with digit set $\mathfrak{N}_q=\{0,1,\ldots,a^2-a-1\}$, and hyperbolic Gaussian integers of the form $q=a+\jg$, with digit set $\mathfrak{N}_q=\{0,1,\ldots,a^2-2\}$, where $a\leq -2$.
 \end{rk}


\subsection{Bicomplex quartic extension}

Let us consider the quartic bicomplex extension
$$
\mathcal{Q}_{\B}:=\Q[\ig,\jg]=\Q+\Q\ig+\Q\jg+\Q\kg .
$$
The bicomplex extension $\mathcal{Q}_{\B}$ contains the imaginary quadratic field $\Q(\ig)$, its copy $\Q(\kg)$, as well as the quadratic hyperbolic extension $\Q_h=\Q[\jg]$ (see Subsection~5.1). Consequently, we may write
$$
\mathcal{Q}_{\B}=\Q(\ig)\eg_1+\Q(\ig)\eg_2=\mathcal{Q}_h+\mathcal{Q}_h\ig .
$$
These decompositions show that $\mathcal{Q}_{\B}$ can be viewed either as the two-dimensional $\Q(\ig)$-algebra $\Q(\ig)[\jg]$ with idempotent basis $(\eg_1,\eg_2)$, or as the free $\mathcal{Q}_h$-module $\mathcal{Q}_h[\ig]$ of rank~$2$ with basis $\{1,\ig\}$. In particular,
$$
\mathcal{Q}_{\B}=K_1\eg_1+K_2\eg_2 \quad \text{with} \quad K_1=K_2=\Q(\ig).
$$
Hence, by Theorem~\ref{decompalg}, $\mathcal{Q}_{\B}$ is generated by a bicomplex number $\omega\notin\C_{\ug}$ for $\ug\in\{\ig,\kg\}$.\\
Since $\mathcal{O}_{\Q(\ig)}=\Z[\ig]$, Theorem~\ref{decompintgers} implies that the ring $\ZB_{\B}$ of integers of $\mathcal{Q}_{\B}$ is a free $\Z$-module of rank~$4$ containing $\Z$, and is given by
$$
\ZB_{\B}=\Z[\ig]\eg_1+\Z[\ig]\eg_2=\ZB_h+\ZB_h\ig ,
$$
with $\{\eg_1,\ig\eg_1,\eg_2,\ig\eg_2\}$ as a $\Z$-basis. Moreover, since $\mathrm{disc}(\Q(\ig)/\Q)=-4$, Proposition~\ref{disriminant} yields
$$
\mathrm{disc}(\mathcal{Q}_{\B}/\Q)=16.
$$
By Proposition~\ref{unitgroup}, the unit group of $\ZB_{\B}$ is finite and satisfies
$$
\ZB_{\B}^{\times}
=\Z[\ig]^{\times}\eg_1+\Z[\ig]^{\times}\eg_2
\backsimeq \Z/4\Z\times\Z/4\Z .
$$Concerning ideals in $\ZB_{\B}$, Proposition~\ref{generideal} together with the fact that $\Z[\ig]$ is a principal ring shows that $\ZB_{\B}$ is also a principal ring. Hence, by Proposition~\ref{unique factorization}, unique factorization holds in $\ZB_{\B}$. In particular, by Proposition~\ref{primes}, the irreducible elements $\pi$ are precisely the primes with $N(\pi)\neq0$, and, up to units, they are of the forms
$$
\pi=\eg_1+g\eg_2
\quad\text{or}\quad
\pi=g\eg_1+\eg_2,
\qquad g \ \text{prime in }\Z[\ig].
$$
It follows that every prime $g$ of $\Z[\ig]$ becomes a semiprime in $\ZB_{\B}$, namely
$$
g=\pi_1\pi_2,
\qquad
\pi_1=\eg_1+g\eg_2,
\quad
\pi_2=g\eg_1+\eg_2 .
$$
Therefore, every prime number $p$ is reducible in $\ZB_{\B}$, and it is semiprime if and only if $p\equiv3\pmod{4}$.

Now let $\omega\in\ZB_{\B}$ and define
\begin{eqnarray*}
A(\omega)&=&\omega\,\overline{\omega}^{\ig}\,\overline{\omega}^{\jg}
+\omega\,\overline{\omega}^{\ig}\,\overline{\omega}^{\kg}
+\omega\,\overline{\omega}^{\jg}\,\overline{\omega}^{\kg}
+\overline{\omega}^{\ig}\,\overline{\omega}^{\jg}\,\overline{\omega}^{\kg},\\
B(\omega)&=&\omega\,\overline{\omega}^{\ig}
+\omega\,\overline{\omega}^{\jg}
+\omega\,\overline{\omega}^{\kg}
+\overline{\omega}^{\ig}\,\overline{\omega}^{\jg}
+\overline{\omega}^{\ig}\,\overline{\omega}^{\kg}
+\overline{\omega}^{\jg}\,\overline{\omega}^{\kg}.
\end{eqnarray*}
Observing that
$$
4\Rj(\omega)=\omega+\overline{\omega}^{\ig}
+\overline{\omega}^{\jg}
+\overline{\omega}^{\kg},
\qquad
N_{\B}(\omega)=\omega\,\overline{\omega}^{\ig}\,
\overline{\omega}^{\jg}\,
\overline{\omega}^{\kg},
$$
a direct computation shows that $\omega$ is a root of
$$
P(X)=X^4-4\Rj(\omega)X^3+A(\omega)X^2-B(\omega)X+N_{\B}(\omega),
$$
where $4\Rj(\omega)$, $A(\omega)$, $B(\omega)$, and $N_{\B}(\omega)$ are rational integers, since they belong to $\ZB_{\B}$ and are invariant under conjugations  (\ref{conjugations}). Hence $P\in\Z[X]$ and the minimal polynomial $P_{\omega}$ divides $P$.\\
In particular, if $\omega\in\Z[\ug]$ for some $\ug\in\{\ig,\kg\}$, or if $\omega\in\ZB_h$, then
$$
P(X)=\bigl(X^2-2\Rj(\omega)X+\omega\overline{\omega}\bigr)^2,
$$
so that $\omega$ is a root of $X^2-2\Rj(\omega)X+\omega\overline{\omega}$, which is its minimal polynomial if and only if $\omega\notin\Z$.

Finally, we turn to the zeta function of the bicomplex extension $\mathcal{Q}_{\B}$. By Proposition~\ref{zeta L},
$$
\zeta_{\mathcal{Q}_{\B}}(s)=\zeta^2_{\Q(\ig)}(s).
$$
Since $\Z[\ig]$ is a principal ring, each ideal $(a)\subseteq\Z[\ig]$ has four associates $\mu a$, $\mu\in\Z[\ig]^{\times}$, of the same norm. Therefore,
$$
a_{\Q(\ig)}(n)=\frac{1}{4}r(n),
$$
where $r(n)$ denotes the number of Gaussian integers of norm~$n$, given explicitly by the well-known Jacobi formula
$$
r(n)=4\Bigl(\sum_{\substack{d\mid n\\ d\equiv1\ (4)}}1
-\sum_{\substack{d\mid n\\ d\equiv3\ (4)}}1\Bigr),
\qquad n\ge1.
$$
Hence, the number $a_{\mathcal{Q}_{\B}}(n)$ of ideals $\fraka\subseteq\ZB_{\B}$ of norm~$n$ is
$$
a_{\mathcal{Q}_{\B}}(n)=\frac{1}{16}\sum_{d\mid n} r(d)\,r\!\left(\frac{n}{d}\right).
$$


\newpage
 \bibliographystyle{alpha}

\end{document}